\documentclass[a4paper,12pt]{article}
\usepackage[utf8]{inputenc}
\usepackage[english]{babel}
\usepackage[ruled,section]{algorithm}
\usepackage{makeidx}
\usepackage[font=small]{caption}
\usepackage{amsmath}
\usepackage{amsthm}
\usepackage{amsfonts}
\usepackage{amssymb}
\usepackage{mathrsfs}
\usepackage{graphicx}
\usepackage[margin=16mm]{geometry}
\usepackage{enumerate}
\usepackage{indentfirst}
\usepackage{color}
\usepackage{authblk}
\makeindex
\newcommand{\real}{\mathbb{R}}

\newcommand{\rN}{ {\mathbb{R}^N} }

\newcommand{\rmd}{\mathrm{d}}

\newcommand{\M}{\mathcal{M}}

\numberwithin{equation}{section}
\newtheorem{theorem}{Theorem}[section]
\newtheorem{lem}[theorem]{Lemma}
\newtheorem{prop}[theorem]{Proposition}

\theoremstyle{definition}

\newtheorem{rmk}[theorem]{Remark}

\begin{document}
\title{\bf\Large Radial solutions for Hénon type fully nonlinear equations in annuli and exterior domains}
\author[1]{Liliane Maia\footnote{l.a.maia@mat.unb.br}}
\author[2]{Gabrielle Nornberg\footnote{gnornberg@dim.uchile.cl}}
\affil[1]{\small Departamento de Matemática, Universidade de Brasília, Brazil}
\affil[2]{\small Instituto de Ciências Matemáticas e de Computação, Universidade de São Paulo, Brazil}

{\date{\today}}

\maketitle

{\small\noindent{\bf{Abstract.}} 
In this note we study existence of positive radial solutions in annuli and exterior domains for a class of nonlinear equations driven by Pucci extremal operators subject to a Hénon type weight.
Our approach is based on the shooting method  applied to the corresponding ODE problem, energy arguments, and the associated flow  of an autonomous quadratic dynamical system.
}
\medskip

{\small\noindent{\bf{Keywords.}} {Fully nonlinear equations; radial positive solutions; dynamical system; shooting method.}
	
	\medskip
	
	{\small\noindent{\bf{MSC2020.}} {35J15, 35J60, 35B09, 34A34.}

\section{Introduction}\label{intro}

In this note we study existence of positive radial solutions of fully nonlinear elliptic partial differential equations in the form
\begin{align}\label{P} 
\left\{
\begin{array}{rclcl}
\mathcal{M}_{\lambda,\Lambda}^\pm (D^2 u)+ |x|^a \,u^p &=&0 &\mbox{in} & \;\Omega \\
u&>& 0 &\mbox{in} & \;\Omega 
\\
u&=& 0 &\mbox{on} & \partial\Omega
\end{array}
\right.
\end{align}
where $p>1$, $a>-1$, $0<\lambda\le \Lambda$, and $\Omega$ is either an annulus or an exterior domain in $\rN$ for $N\geq 3$. Additionally, we suppose
 $\tilde N_+>2$ in the case of $\M^+_{\lambda,\Lambda}$, where $\tilde N_\pm$ are the dimension-like numbers
\begin{center}
	$\tilde N_+ = \frac{\lambda}{\Lambda} (N-1)+1$, \qquad $\tilde N_- = \frac{\Lambda}{\lambda} (N-1)+1$.
\end{center}
Here $\mathcal{M}_{\lambda,\Lambda}^\pm$ are the Pucci's extremal operators, which play an essential role in stochastic control theory and mean field games. 
We deal with classical solutions of \eqref{P} that are $C^2$ in $\Omega$.

\smallskip

In \cite{MNPscalar}  nonexistence results in exterior domains for weighted equations as in \eqref{P} via a dynamical system approach were established. However, proving the existence of such solutions, as well as solutions in annuli, is a difficult task in terms of that approach, since the orbits in the flow there present blow-up discontinuities.
Our goal in this work is to complement the analysis in \cite{MNPscalar}, by showing existence of radial solutions in exterior domains and annuli.

We mention that the analysis of the associated ODE problem for proving existence of annular or exterior domain solutions has been performed in many papers in semilinear cases \cite{BCManel87, LinPai91, NiNussbaum}.

In the case when $a=0$, results of this nature were obtained in \cite{GLPradial, GILexterior2019}. 
The analysis in \cite{GILexterior2019} was performed in light of the change of variables in \cite{FQTaihp}. In \cite{MNPscalar} we have already noticed that employing quadratic dynamical systems is effective to deal with these problems, in a simple and unified way.
Here, our shooting arguments to obtain existence of annular solutions are inspired by those in \cite{GLPradial}. On the other hand, in what concerns the existence in exterior domains we introduce an alternative dynamical system, which is different from those in \cite{GILexterior2019} and \cite{MNPscalar}, of quadratic type which do not present blow-up discontinuities.

\smallskip

For solutions in annuli,  our main result reads as follows.
\begin{theorem}\label{teo anel}
For any $p>1$, and $0<\frak a< \frak b<+\infty $, problem \eqref{P} has a positive  radial solution in the annulus\vspace{-0.1cm}
\begin{align*}
\textrm{	$\Omega =\{ x\in \rN \, :\, \frak{a}<|x|<\frak{b}\}$, \; with $0<\frak a<\frak b< +\infty$. }
\end{align*}
\end{theorem}
Note that solutions of \eqref{P} may not be radial in general, for instance the Gidas-Ni-Nirenberg type symmetry result of \cite{BdaLioSym} does not hold for annular domains.
Moreover, since $v=-u$ solves $\mathcal{M}^\mp_{\lambda,\Lambda} (v)+|x|^a \,|u|^{p-1}u$ in $\Omega$, then Theorem \ref{teo anel} also proves the existence of a negative solution for the corresponding problem.

The proof of Theorem \ref{teo anel} relies on a careful study of the ODE problem, shooting method and energy arguments. 
We mention that solutions in annuli can be identified in correspondence with orbits in  the dynamical system, but the interior and exterior radii are not explicit in that approach neither can be obtained for an arbitrary annulus via rescaling.

As far as exterior domain solutions are concerned, we obtain the following result. We recall $p^*_{a\pm}$ are the critical exponents defined in \cite{MNPscalar} for the operators $\M^\pm_{\lambda,\Lambda}$. They are the threshold for the existence and nonexistence of radial positive regular solutions, that is, solutions differentiably defined at $r=0$.
\begin{theorem}[Exterior domain]\label{teo exterior}
For any $p>p^{*}_{a\pm}$ and $R>0$, by setting $\Omega =\rN\setminus B_R$, it holds:
\begin{enumerate}[(i)]
\item there exists a unique fast decaying solution of \eqref{P};

\item there exist infinitely many solutions of \eqref{P} with either slow or pseudo-slow decay.
\end{enumerate}
\end{theorem}
In addition, the ranges where pseudo-slow exterior domain solutions exist are the same where pseudo-slow decay regular solutions exist in \cite{MNPscalar}, see also \cite{PacellaStolnicky1}.

The choice of using quadratic systems to treat weighted equations is categorical since the new dynamical system variables do not see the weight, as in \cite{MNPscalar}. It is worth mentioning that earlier methods employed might be much more involved, meanwhile a simple variable change which eliminates the weight is not available for Pucci operators.

The text is organized as follows. 
In Section \ref{preliminaries} we recall some preliminary tools on radial solutions and shooting method. 
In Section \ref{section exterior} we write the corresponding quadratic system and use it to prove Theorem \ref{teo exterior}, while Section \ref{section annuli} is devoted to the proof of Theorem \ref{teo anel}.

\section{Preliminaries}\label{preliminaries}

We start by recalling that Pucci's extremal operators $\mathcal{M}^\pm_{\lambda,\Lambda}$, for $0<\lambda\leq \Lambda$, are defined as
$$
\textstyle{\mathcal{M}^+_{\lambda,\Lambda}(X):=\sup_{\lambda I\leq A\leq \Lambda I} \mathrm{tr} (AX)\,,\quad \mathcal{M}^-_{\lambda,\Lambda}(X):=\inf_{\lambda I\leq A\leq \Lambda I} \mathrm{tr} (AX),}
$$
where $A,X$ are $N\times N$ symmetric matrices, and $I$ is the identity matrix.
Equivalently, if we denote by $\{e_i\}_{1\leq i \leq N}$ the eigenvalues of $X$, we can define the Pucci's operators as
\begin{align}\label{def Pucci}
\textstyle{\textrm{$\mathcal{M}_{\lambda,\Lambda}^+(X)=\Lambda \sum_{e_i\ge 0} e_i +\lambda \sum_{e_i<0} e_i$, \;\;\; $\mathcal{M}_{\lambda,\Lambda}^-(X)=\lambda \sum_{e_i \ge 0} e_i +\Lambda \sum_{e_i<0} e_i$}. }
\end{align}
From now on we will drop writing the parameters $\lambda,\Lambda$ in the notations for the Pucci's operators.

When $u$ is a radial function, for ease of notation we set $u(|x|)=u (r)$ for $r=|x|$. If $u$ is $C^2$, the eigenvalues of the Hessian matrix $D^2 u$ are $\{u^{\prime\prime}, \frac{u^\prime (r)}{r}, \ldots, \frac{u^\prime (r)}{r}\}$ where $\frac{u^\prime (r)}{r}$ is repeated $N-1$ times.  

The Lane-Emden system \eqref{P} for $\mathcal{M}^+$ is written in radial coordinates as
\begin{align}\label{P radial m}\tag{$P_+$}
\begin{array}{l}
u^{\prime\prime}\;=\; M_+(-r^{-1}(N-1)\, m_+(u^\prime)-r^a u^p ), \quad u> 0  ,
\end{array}
\end{align}
while for $\M^-$ one has
\begin{align}\label{P radial m-}\tag{$P_-$}
\begin{array}{l}
u^{\prime\prime}\;=\; M_-(-r^{-1}(N-1)\, m_-(u^\prime)-r^a u^p ), \quad u> 0  ,
\end{array}
\end{align}
where $M_\pm$ and $m_\pm$ are the Lipschitz functions
\begin{align}\label{m,M+}
	m_+(s)=
	\begin{cases}
		\lambda s\; \textrm{ if } s\leq 0 \\
		\Lambda s\; \textrm{ if } s> 0
	\end{cases}\;
	\textrm{and}\quad
	M_+(s)=
	\begin{cases}
		s/\lambda\; \textrm{ if } s\leq 0 \\
		s/ \Lambda\; \textrm{ if } s> 0,
	\end{cases}
\end{align}
\vspace{-0.6cm}
\begin{align}\label{m,M-}
	m_-(s)=
	\begin{cases}
		\Lambda s\; \textrm{ if } s\leq 0 \\
		\lambda s\; \textrm{ if } s> 0
	\end{cases}\;
	\textrm{and}\quad
	M_-(s)=
	\begin{cases}
		s/\Lambda\; \textrm{ if } s\leq 0 \\
		s/ \lambda\; \textrm{ if } s> 0.
	\end{cases}
\end{align}

Equations \eqref{P radial m} and \eqref{P radial m-} are understood in the maximal interval where $u$ is positive.

\begin{rmk}\label{rmk maximum point}
A positive function $u$ cannot be at the same time convex and increasing in an interval, since $u^{\prime\prime}(r)<0$ as long as $u'(r)\geq 0$. In particular, any critical point of a positive solution $u$ is a local strict maximum point for $u$. 
\end{rmk}

By solution in annulus or exterior domain solution we mean a solution $u$ of \eqref{P radial m} or \eqref{P radial m-} defined in an interval $[\frak a, \rho)$, for $\frak a\in (0, +\infty)$ and $\rho\leq +\infty$, and verifying the Dirichlet condition $u(\frak a)=\lim_{r\to \rho^-}u(r)=0$. We look at the initial value problem
\begin{align}\label{shooting der}
\begin{cases}
u^{\prime\prime}\;=\; M_\pm\left( -r^{-1}(N-1)\, m_\pm(u^\prime)- r^a |u|^{p-1}u \,\right),\smallskip\\
\quad u (\frak a)=0  , \;\; u^\prime (\frak a)=\delta , \qquad \delta>0 ,
\end{cases}
\end{align}

The equations \eqref{P radial m}, \eqref{P radial m-}, together with \eqref{shooting der} were studied in \cite{GILexterior2019}, \cite{GLPradial}. 

\medskip

For any $p>1$ and for each $\delta>0$, by ODE theory there exists a unique solution $u=u_\delta$ defined in a maximal interval $(\frak a , \rho_\delta)$ where $u$ is positive, $\frak a <\rho_\delta\leq +\infty$.

If $\rho_\delta=+\infty$ we get a positive radial solution in the exterior of the ball $B_{\frak a}$. In the second case, a positive solution in the annulus $(\frak a,\rho_\delta)$ is produced.
Note that equations \eqref{shooting der} are not invariant by rescaling.

\begin{rmk}\label{rmk G}
	All results obtained for $\delta>0$ will be also true for $\delta<0$. Indeed, negative shootings for an operator $F$ can be seen as positive shootings for the operator $G$ defined as $G(x,X)=-F(x,-X)$, which is still elliptic and satisfies all the properties we considered so far. In particular, the negative solutions of $\mathcal{M}^+$ are positive solutions of $\mathcal{M}^-$ in the same domain, and viceversa.	
\end{rmk}

\begin{rmk}\label{rescaling}		
	Let $u$ be a positive radial solution of \eqref{shooting der} with $u=u_{\delta}$, for some positive constant $\delta$ in $[\frak{a},\rho)$. Then the rescaled pair
	\begin{align}\label{eq scaling}
	\textrm{${u}_\gamma =\gamma\, u(\gamma^{\frac{1}{\alpha}} r)$ , \qquad \,$\gamma>0$,}
	\end{align}
for $\alpha$ as in \eqref{critical exponents a}, produces a positive solution pair of the same equation in $({\frak a}{\gamma^{-\frac{1}{\alpha}}}, {\rho}{\gamma^{-\frac{1}{\alpha}}})$ with initial values $u_\gamma
({\frak a}\gamma^{-\frac{1}{\alpha}} )
=0$, as well as $u_\gamma^\prime({\frak a}\gamma^{-\frac{1}{\alpha}} )=\gamma^{1+\frac{1}{\alpha}} \,\delta$.
\end{rmk}

\textbf{Notation.} Whenever $p>1$ and $\tilde N_+>2$, we set
\begin{align}\label{critical exponents a}
\textstyle {p^{p,a}_\pm=\frac{\tilde{N}_\pm+2a+2}{\tilde{N}_\pm-2},	\quad p^{s,a}_\pm=\frac{\tilde{N}_\pm+a}{\tilde{N}_\pm - 2},	\quad p_\Delta^a=\frac{N+2+2a}{N-2},	\quad \alpha =\frac{2+a}{p-1}.}
\end{align}

\section{The dynamical system and exterior domain solutions}\label{section exterior}

Let $u,v$ be a positive solution pair of \eqref{P radial m} or \eqref{P radial m-}. Thus we can define the new functions
\begin{align}\label{X,Y,Z,W sem a}
x(t)=-\frac{ru^\prime}{u}, \qquad  z(t)={r^{2+a}\, u^{p-1}},  
\end{align}
for $t=\mathrm{ln}(r)$, whenever $r>0$ is such that $u> 0$.
The phase space is contained in $ \real^2$.
Throughout the text, we denote the first quadrant as
\begin{center}
	$1Q=\{\,(x,z)\in \real^2:\; x,z>0\,\}$.
\end{center}

Since we are studying positive solutions, the points $(x(t),z(t))$ belong to $1Q$ when $u^\prime<0$. Apart from $1Q$, we set
\begin{center}
	$2Q=\{\,(x,z)\in \real^2:\; x<0,\;\; z>0\,\}$,
\end{center}
that is, $2Q$ is the region in $\real^2$ such that the corresponding $u$ satisfies $u^\prime>0$. In particular, $u^{\prime\prime}<0$ in $2Q$.

As a consequence of this monotonicity, the problems \eqref{P radial m} and \eqref{P radial m-} become in $1Q$ as:
\begin{align}\label{P M+ 1Q}
\textrm{for $\M^+$ in $1Q$} :\quad
u^{\prime\prime}\;=\; M_+(-\lambda r^{-1}(N-1)\, u^\prime-r^a u^p ), \quad u> 0  ;
\end{align}
\vspace{-0.5cm}
\begin{align}\label{P M- 1Q}
\textrm{\;\;for $\M^-$ in $1Q$} :\quad 
u^{\prime\prime}\;=\; M_-(-\Lambda r^{-1}(N-1)\, u^\prime- r^a v^p ),  \quad u> 0  .
\end{align}

In terms of the functions \eqref{X,Y,Z,W sem a}, we derive the following autonomous dynamical system, corresponding to \eqref{P M+ 1Q} for $\M^+$, where the dot $\dot{ }$ stands for $\frac{\mathrm{d}}{\mathrm{d} t}$,
\begin{align}\label{DS+}
\textrm{$\M^+$ in $1Q$ : \quad}
\left\{
\begin{array}{ccl}
\dot{x} &=& \;x\, (x+1)-M_+ (\lambda(N-1)x-z )\, \smallskip \\
\dot{z} &=& \;z\,(x+2+a-px) .
\end{array}
\right.
\end{align}
Likewise one has for $\M^-$, associated to \eqref{P M- 1Q}, 
\begin{align}\label{DS-}
\textrm{$\M^-$ in $1Q$ : \quad}
\left\{
\begin{array}{ccl}
\dot{x} &=& \;x\, (x+1)-M_- (\Lambda(N-1)x-z )\,\smallskip  \\
\dot{z} &=& \;z\,(x+2+a-px) .
\end{array}
\right.
\end{align}
We stress that \eqref{DS+} and \eqref{DS-} correspond to positive, decreasing solutions of \eqref{P radial m} and \eqref{P radial m-}. 
\smallskip

On the other hand, given a trajectory $\tau=(x,z)$ of \eqref{DS+} or \eqref{DS-} in $1Q$, we define
\begin{align}\label{def u via X,Z}
\textstyle{ u(r)=r^{-\alpha}\, z^{\frac{1}{p-1}}}, \;\;\textrm{ where }\, r=e^t,
\end{align}
and then we deduce
\begin{align*}
u^\prime (r) \textstyle = -\alpha r^{-\alpha-1} \,z^{\frac{1}{p-1}}(t) 
	+ \frac{r^{-\alpha}}{p-1} \,z^{\frac{1}{p-1}-1}(t) \,\frac{\dot{z} }{r} 	
	= \frac{u}{r}\{ -\alpha +\frac{x+2+a-px}{p-1} \}=-\frac{x(t)u(r)}{r}.
\end{align*}
Since $x\in C^1$, then $u\in C^2$. Moreover, $u$ satisfies either \eqref{P radial m} or \eqref{P radial m-} from the respective equations for $\dot{x},\dot{z}$ in the dynamical system.

\smallskip	

In other words, $(x,z)$ is a solution of equation \eqref{DS+} or \eqref{DS-} in $1Q$ if and only if $u$ defined by \eqref{def u via X,Z} is a positive pair solution of \eqref{P radial m} or \eqref{P radial m-} with $u^{\prime}<0.$ 

\smallskip

An important role in the study of our problem is played by the following line for $\M^+$,
\begin{align}\label{pi M+}
\ell_+=\{\,(x,z): z=\lambda (N-1)x\,\}\cap 1Q,
\end{align}
which corresponds to the vanishing of $u^{\prime\prime}$, see also \cite{MNPscalar}. It allows us to define the following regions 
\begin{align}\label{R+- M+}
R^+_{\lambda}=\{(x,z)\in 1Q: z>\lambda (N-1)x\},\quad R^-_{\lambda}=\{(x,z)\in 1Q: z<\lambda (N-1)x\},
\end{align}
which represent the sets where the function $u$ is concave or convex.
More precisely, $R^+_{\lambda}$ is the region of strictly concavity of $u$, while $R^-_{\lambda}$ is the region of strictly convexity of $u$.

The respective notations for the operator $\M^-$ are 
\begin{align}\label{pi M-}
\ell_-=\{\,(x,z): z=\Lambda (N-1)x\,\}\cap 1Q,
\end{align}
\vspace{-0.8cm}
\begin{align}\label{R+- M-}
R^+_{\Lambda}=\{(x,z)\in 1Q: z>\Lambda (N-1)x\},\quad
R^-_{\Lambda}=\{(x,z)\in 1Q: z<\Lambda (N-1)x\}.
\end{align}

At this stage it is worth observing that the systems \eqref{DS+} and \eqref{DS-} are continuous on $\ell_+$ and $\ell_-$, respectively. More than that, the right hand sides are locally Lipschitz functions of $x,z$, so the usual ODE theory applies. 
Then one recovers existence, uniqueness, and continuity with respect to initial data as well as continuity with respect to the parameter $p$.

Since we are considering positive solutions of \eqref{P}, and $u^\prime>0$ implies $u^{\prime\prime}<0$, one finds out the following ODEs:
\begin{align}\label{P M+ 2Q}
\textrm{\;\;for $\M^+$ in $2Q$} :\quad \left\{\begin{array}{l}
\lambda u^{\prime\prime}\;=\; -\Lambda r^{-1}(N -1)u^\prime-r^a  u^p ,  \quad u> 0  ;\end{array}\right.
\end{align}\vspace{-0.5cm}
\begin{align}\label{P M- 2Q}
\textrm{\;\;for $\M^-$ in $2Q$} :\quad \left\{\begin{array}{l}
\Lambda u^{\prime\prime}\;=\; -\lambda r^{-1}(N -1)u^\prime- r^a u^p , \quad u> 0; \end{array}\right.
\end{align}
Now, in terms of the corresponding dynamical system, we get 
\begin{align}\label{DS+ 2Q}\textrm{$\M^+$ in $2Q$\,:\;\;}
\left\{\begin{array}{cc}
\dot{x} = \;x \, (x-\tilde{N}_- +2) \,+\frac{z}{\lambda} , \smallskip \\ \dot{z}= \;z\, (\, x+2+a-px\,). \;
\end{array}\right.
\end{align}

On the other hand, for the operator $\M^-$ one has
\begin{align}\label{DS- 2Q}\textrm{$\M^-$ in $2Q$\,:\;\;}
\left\{\begin{array}{cc}
\dot{x} = \;x \, (x-\tilde{N}_+ +2) \,+\frac{z}{\Lambda} , \smallskip \\ \dot{z}= \;z\, (\, x+2+a-px\,). \;
\end{array}\right.
\end{align}

\begin{rmk}
No stationary points exist in $2Q$. Indeed, since $u^\prime >0$ and $u>0$ yield $u^{\prime\prime}<0$ then $\dot{x}>0$ in $2Q$. In other words, we do not have stationary points outside $1Q$ when we are considering positive solutions $u$ of \eqref{P}.
\end{rmk}

One may write the dynamical systems \eqref{DS+} and \eqref{DS-} in terms of the following ODE first order autonomous equation
\begin{align}\label{Prob x=(X,Z)}
(\dot{x}, \dot{z})=(\, f(x,z) , \, g(x,z)\, ). 
\end{align}
For instance, in the case of the operator $\M^+$, then $f,g$ are given by
\begin{center}
	$f(x,z)=
	\begin{cases}
x\, (x-N+2)+\frac{z}{\lambda} \;\textrm{ in } R^+_\lambda\\
x\, (x-\tilde{N}_+ +2)+\frac{z}{\Lambda}\;\textrm{ in } R^-_\lambda
	\end{cases}
	,\quad
	g(x,z)=z\,(x+2+a-px)$.
\end{center}

We first recall some standard definitions from the theory of dynamical systems.

A stationary point $Q$ of \eqref{Prob x=(X,Z)} is a zero of the vector field $(f,g)$.
	If $\sigma_1$ and $\sigma_2$ are the eigenvalues of the Jacobian matrix $(Df(Q),Dg(Q))$, then $Q$ is hyperbolic if both $\sigma_1, \sigma_2$ have nonzero real parts. If this is the case, $Q$ is a \textit{source} if $\mathrm{Re}(\sigma_1),\mathrm{Re}(\sigma_2)>0$, and a \textit{sink} if $\mathrm{Re}(\sigma_1),\mathrm{Re}(\sigma_2)<0$; $Q$ is a \textit{saddle point} if $\mathrm{Re}(\sigma_1)<0<\mathrm{Re}(\sigma_2)$.

Next we recall a result on the local stable and unstable manifolds near saddle points of the system \eqref{Prob x=(X,Z)}; see \cite[theorems 9.29, 9.35]{Hale}.
We will see that the usual theory for autonomous planar systems applies since each stationary point $Q$ possesses a neighborhood strictly contained in $R^-_\lambda$ in the first quadrant where the function $f$ is $C^1$. In turn, $g$ is always a $C^1$ function.

Sometimes we denote $\alpha(\tau)$, i.e.\ the $\alpha$-limit of the orbit $\tau$, as the set of limit points of $\tau(t)$ as $t\to -\infty$.
	Similarly one defines $\omega(\tau)$ i.e.\ the $\omega$-limit of $\tau$ at $+\infty$.

We observe that the $x$ axis is invariant by the flow in the sense that $z=0$ implies $\dot z=0$.
Next, the set where $\dot{x}=0$ for the system \eqref{DS+}, with respect to the operator $\M^+$, is given by the parabola
\begin{align}\label{ell 1+}
\pi_1^+ = \{(x,z):\, z = \Lambda (\tilde{N}_+-2)x-\Lambda x^2\}\cap 1Q.
\end{align}
Note that a respective parabola $\{\,(x,z):\, z = \lambda (N-2)x-\lambda x^2\,\}$ where $\dot x =0 $ on the region $R_\lambda^+$ does not exist, since it lies entirely below the concavity line $\ell_+$, namely
\begin{center}
$z = \lambda (N-2)x-\lambda x^2 <  \lambda (N-2)x <\lambda (N-1)x$\;\; for $x>0$.
\end{center}
Moreover , the parabola $\pi_1^+$ itself in \eqref{ell 1+} lies below the concavity line $\ell_+$, so belonging to the region $R_\lambda^-$, that is,
\begin{center}
	$z = \Lambda (\tilde{N}_+-2)x-\lambda x^2 <  \Lambda (\tilde{N}_+-2)x <\lambda (N-1)x$\;\; whenever $x>0$,
\end{center}
since $\tilde{N}_+-2<\tilde{N}_+-1=\frac{\lambda}{\Lambda} (N-1)$.
Analogously, for the operator $\M^-$, the parabola
\begin{align}\label{ell 1-}
\pi_1^- = \{\,(x,z):\, z = \lambda (\tilde{N}_--2)x-\lambda x^2\,\}\cap 1Q
\end{align}
represents the set where $\dot{x}=0$, which is contained in the region $R_\Lambda^-$. 
Also, we define the line
\begin{align}\label{ell 2+}
\pi_2=\pi_2^\pm = \{\, (x,z):\, x=\alpha\, \}\cap 1Q,
\end{align}
which is the set where $\dot{z}=0$ and $z>0$ for both operators $\M^\pm$.

\begin{lem}\label{stationary M+}
	The stationary points of the dynamical systems \eqref{DS+}, \eqref{DS-}, \eqref{DS+ 2Q}, and \eqref{DS- 2Q} are:
	\begin{center}
		for $\M^+$:\quad	$O=(0,0)$,  \quad $A_0=(\tilde{N}_+-2,0)$, \quad $M_0=(x_0,z_0)$,
	\end{center}
	where $x_0 = \alpha$, and $z_0 =\alpha\Lambda(\tilde{N}_+-p\alpha+a) = \alpha\Lambda(\tilde{N}_+-2 - \alpha)$;
	\begin{center}
		for $\M^-$:\quad	$O=(0,0)$, \quad $A_0=(\tilde{N}_--2,0)$, \quad $M_0=(x_0,z_0)$,
	\end{center}
	where $x_0 = \alpha$ and $z_0 = \alpha \lambda(\tilde{N}_--p\alpha+a) = \alpha\lambda(\tilde{N}_--2 - \alpha)$.
\end{lem}

\begin{proof} 
	The stationary points are given by the intersection of the parabola $\pi_1^\pm$ with the lines $\pi_2$ and $\{(x,z)\in \real^2: z=0\}$.
\end{proof}

Next we analyze the directions of the vector field $F$ in \eqref{Prob x=(X,Z)} on the $x,z$ axes, on the concavity lines
$\ell_\pm$, and on the sets $\pi_1^\pm$ and $\pi_2$.

\begin{prop}\label{prop flow} The systems \eqref{DS+} and \eqref{DS-} enjoy the following properties:	
	\begin{enumerate}[(1)]
		\item Every trajectory of \eqref{DS+} in $1Q$ crosses the line $\ell_+$ transversely except at the point
		$P=\frac{1+a}{p}(1, \lambda (N-1))$. 
		It passes from 
		$R^+_\lambda$ to $R^-_\lambda$  if $x> \frac{1+a}{p}$, while it moves from $R^-_\lambda$ to $R^+_\lambda$ if $x< \frac{1+a}{p}$.
A similar statement holds for \eqref{DS-}, via $\ell_-$, $P=\frac{1+a}{p}(1, \Lambda (N-1))$, $R^+_\Lambda$, $R^-_\Lambda$;
		
		\item The vector field at the point $P$ in item (1) is parallel to the line $\ell_+$ (resp.\ $\ell_-$), and  an orbit can only reach such a point from $R^-_\lambda$ (resp.\  $R^-_\Lambda$);
		
		\item The flow induced by \eqref{Prob x=(X,Z)} on the $x$ axis points to the left for $x\in (0, \tilde{N}_\pm-2)$, and to the right when $x> \tilde{N}_\pm-2$. On the $z$ axis it always moves up and to the right;
		
		\item The vector field $(f,g)$ on the parabola $\pi_1^\pm$ is parallel to the $z$ axis whenever $x\neq \alpha$. It points up if $x< \alpha$, and down if $x>\alpha$. 
		
		\item On the line $\pi_2$ the vector field $(f,g)$ is parallel to the $x$ axis for $z\neq z_0$, where $z_0$ is the $z$-coordinate of $M_0$ in Lemma \ref{stationary M+}. It moves to the left if $z< z_0$, and to the right if $z>z_0$.
	\end{enumerate}
	
\end{prop}

\begin{proof} Let us consider the operator $\M^+$ since for $\M^-$ it will be analogous.
	
\textit{(1)--(2)} We have $\dot{x}>0$ on $\ell_+$ since $\pi_1^+$ is below the line $\ell_+$. 
By the inverse function theorem, $t$ is a function of $x$ on $\ell_+$, and so is $z$. 
In order to detect the transversality, one looks at $\frac{\mathrm{d}z}{\mathrm{d}x}$ on $\ell_+$ and compare it with the slope of $\ell_+$. 
Note that
\begin{align}\label{Z(X)}
\textstyle \frac{\mathrm{d}z}{\mathrm{d}x} = \frac{\dot{z}}{\dot{x}} =  \frac{\lambda (N-1)(p-1)x(\alpha-x)}{x(x+1)} \quad \textrm{on\,\; $\ell_+$}.
\end{align}
Then $\frac{\mathrm{d}z}{\mathrm{d}x}>\lambda (N-1)$ for $x<\frac{1+a}{p}$, while $\frac{\mathrm{d}z}{\mathrm{d}x}<\lambda (N-1)$ for $x>\frac{1+a}{p}$.\smallskip

Now we infer that a change of concavity does not happen at $x=\frac{1+a}{p}$. Indeed, if $\tau=(x,z)$ is a trajectory and
$t_0$ is such that $\tau(t_0)=P$, then $\frac{\mathrm{d}z}{\mathrm{d}x}(\frac{1+a}{p})=\lambda (N-1)$ i.e. $\frac{\mathrm{d}\tilde{z}}{\mathrm{d}x}(\frac{1+a}{p})=0$, where $\tilde{z}(x)=z(x)-\lambda(N-1)x$.
Then $\tilde{z}$ has a maximum point at $x=\frac{1+a}{p}$, and so $\tau$ stays in $R^-_\lambda$ in a neighborhood of the point $P$.
	
\textit{(3)}  Since the $x$ axis is contained in $R^-_\lambda$ then $\dot{x}=x(x-(\tilde{N}_+-2))$ which is positive for $x>\tilde{N}_+-2$ and negative when $x<\tilde{N}_+-2.$
On the other hand, the $z$ axis is contained in $R^+_\lambda$, so $\dot x = \frac{z}{\lambda}>0$ and $\dot{z}= z(2+a)>0$ on $x=0$ for $a>-2$.
	
\textit{(4)}
	$\dot{z}=(p-1)z(\alpha-x)$ on $\pi_1^+$ is positive for $x<\alpha$, and negative when $x> \alpha$. 

\textit{(5)} On $\pi_2\cap R_\lambda^-$, $\dot{x} =\alpha (\alpha - (\tilde{N}_+-2))+\frac{z}{\Lambda}$  is positive for $z>z_0$ and negative when $z<z_0$,  for $z_0$ given in Lemma \ref{stationary M+}. On $\pi_2 \cap  (R_\lambda^+\cup \ell_+)$ we have $\dot{x}$$ \ge \alpha (\alpha +1)>0$. 
\end{proof}
\vspace{-0.4cm}
\begin{figure}[!htb]\centering	\includegraphics[scale=0.7]{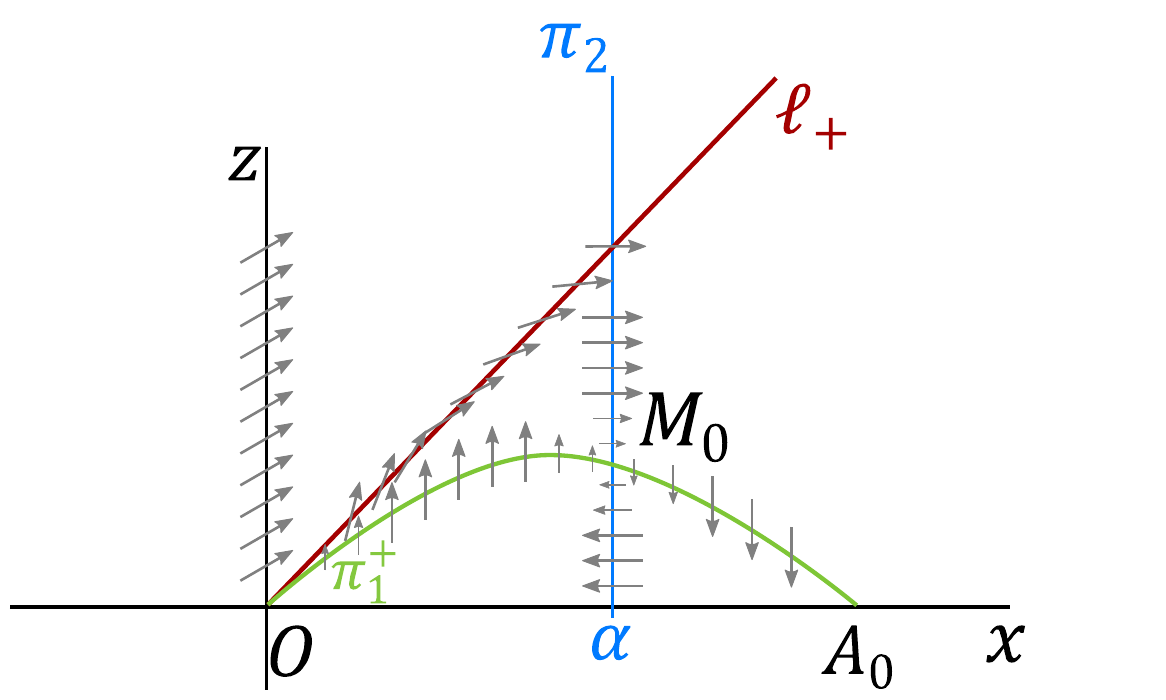}
	\caption{The phase plane $(x,z)$ when $p>p^{s,a}_+$.}
	\label{Fig flow}
\end{figure}

The next proposition gives us a study on a local analysis of the stationary points. It follows by the correspondence with the variables $(X,Z)$ in \cite{MNPscalar}, and Propositions 2.2, 2.7, and 2.8, in addition to the appendix in there. 

We recall that $X:=x$ and $Z(t):=-{r^{1+a}u^p(r)}/{u^\prime(r)}$, $r=e^t$ from \cite{MNPscalar}. Note that $z=XZ$. 
\begin{prop}[$\M^\pm$]\label{local study stationary points}
The following properties are verified for the systems \eqref{DS+} and \eqref{DS-}.
\begin{enumerate}
\item For every $p>1$ the origin $O$ is a saddle point, whose unstable direction is given by 
\begin{align*}
\textrm{$z=\Lambda (\tilde{N}_++a)x $\; if the operator is $\M^+$,\quad 
$z= \lambda (\tilde{N}_- +a)x$\; for $\M^-$.}
\end{align*}
Moreover, there is a unique trajectory coming out from $O$ at $-\infty$ with slope as above, which we denote by $\Gamma_p$\,, that is,
for all $p>1$, $\Gamma_p$ is such that  $\alpha(\Gamma_p)=O$.
		
\item For $p>p^{s,a}_\pm$ the point $A_0 $ is a saddle point whose linear stable direction is 
\begin{align*}
z=-A_\pm \,x +A_\pm (\tilde N_\pm -2) ,
\end{align*}
where $A_+ =\Lambda [(\tilde{N}_+-2) p - (2+a)]$ for $\M^+$, while $A_- =\lambda [(\tilde{N}_--2) p - (2+a)]$ for $\M^-$.
Further, there exists a unique trajectory arriving at $A_0$ at $+\infty$ with slope as above, which we denote by $\Upsilon_{p}$ i.e.\
for all $p>p^{s,a}_\pm$,   $\Upsilon_p$ is such that\quad $\omega(\Upsilon_p)=A_0$.

\item
At $p=p^{s,a}_\pm$ the point  $A_0$ coincides with $M_0$ and belongs to the $x$ axis, while for $p< p_\pm^{s,a}$ the point $A_0$ is a source and $M_0$ belongs to the fourth quadrant. Also,  $M_0\in 1Q \,\Leftrightarrow\, p>p^{s,a}_\pm$ in which case: $M_0 $ is a source if $p_\pm^{s,a}<p<p_\pm^{p,a}$; $M_0$ is a sink for $p>p_\pm^{p,a}$;
$M_0$ is a center at $p=p^{p,a}_\pm$.
\end{enumerate}
\end{prop}

The trajectories $\Gamma_p$ and $\Upsilon_p$ uniquely determine the global unstable and stable manifolds of the stationary points $O$ and $A_0$, respectively.  They are graphs of functions in a neighborhood of the stationary points in their respective ranges of $p$.
The tangent direction at $O$ is always above the parabola $\pi_1^\pm$ with $\dot z >0$, while the tangent direction at $A_0$ is above $\pi_1^\pm$ with $\dot z <0$ for all $p> p^{s,a}_\pm$.

\medskip

Next we translate the results obtained in \cite{MNPscalar} into the new variables in what concerns periodic orbits, a priori bounds and blow-ups. We use the  one to one correspondence between the orbits in the system $(X,Z)$ in \cite{MNPscalar} and our $(x,z)$. 

Since $\dot x>0$ on $2Q$ then periodic orbits may only exist in $1Q$. From this, we automatically recover the following proposition from \cite{MNPscalar}.
Recall that $p^{p,a}_-<p_\Delta^a < p^{p,a}_+$ from \eqref{critical exponents a}.

\begin{prop}[Dulac's criterion]\label{Dulac}
	Let $\lambda<\Lambda$. 
In the case of the operator $\M^+$ there are no periodic orbits of \eqref{DS+} when $1<p\leq  p_\Delta^a$ or $p> p^{p,a}_+$. In the case of $\M^-$ no periodic orbits of \eqref{DS-} exist if $1<p<p^{p,a}_-$ or $p\geq p_\Delta^a$. In addition, for $\M^+$
,	\vspace{-0.2cm}
	\begin{enumerate}[(i)]
		\item there are no periodic orbits strictly contained in the region $R^+_\lambda\cup \ell_+$ $($resp.\ $R^+_\Lambda\cup \ell_-$ for $\M^-)$, for any $p>1$;
		\vspace{-0.2cm}
		\item periodic orbits contained in $R^-_\lambda\cup \ell_+$ $($resp.\ $R^-_\Lambda\cup \ell_-)$ are admissible only at $p=p^{p,a}_\pm$. 
		Also, no periodic orbits at $p^{p,a}_\pm$ can cross the concavity line $\ell_\pm$ twice;
		\vspace{-0.2cm}
		\item other limit cycles $\theta$ are admissible by the dynamical system as far as they cross $\ell_\pm$ twice.
	\end{enumerate}
\end{prop}

By Poincaré-Bendixson theorem, if a trajectory of \eqref{DS+} or \eqref{DS-} does not converge to a stationary point neither to a periodic orbit, either forward or backward in time, then it necessarily blows up. 
In the next propositions we prove that a blow up may only occur in finite time. The admissible blow-ups for $x$ in forward time are again in correspondence with $X$ in \cite{MNPscalar}.
However, blow-ups in $Z$ in $1Q$ from \cite{MNPscalar} do not occur in our system for $z$, since a blow-up there occurs at a finite time $T$ where $u^\prime (T)=0$, $R=e^{T}$, and for this we have $z(T)=R^{2+a}u^{p-1}(R)\in (0,+\infty)$.

It remains to characterize the types of blow-up for $x$ in $2Q$ backward in time.
\begin{lem}\label{lemma 2Q}
Any trajectory $\tau$ of \eqref{DS+}--\eqref{DS+ 2Q} or \eqref{DS-}--\eqref{DS- 2Q} which passes through $2Q$, with $\tau(t)=(x(t),z(t))$, is such that $x(t)\to -\infty$ and $z(t)\to 0$ as $t\to t_1$ for some $t_1\in \real$.
Also, the vector field in $2Q$ always points to the right and upwards, with $\dot{x}>0$ and $\dot{z}>0$.
\end{lem}

\begin{proof} Let us consider the operator $\M^+$. In $2Q$ one uses the systems \eqref{DS+ 2Q} to write
\begin{center}
$\dot x = x(x+2-\tilde{N}_-)+\frac{z}{\lambda}>0$, \quad  $\dot z = z(p-1)(\alpha -x)>0$,
\end{center}
since $x<0$ and $z>0$.
Moreover, if $x(t_0)<0$ for some $t_0\in \real$, we write	
$\frac{\dot{x}}{x(x+2-\tilde{N}_-)}\geq 1$ and so
\begin{align*}
\textstyle{\frac{\rmd}{\rmd t}\,\mathrm{ln} \left(  \frac{x(t)+2-\tilde{N}_-}{x(t)}  \right)
	=\frac{\dot{x}}{x-(\tilde{N}_--2)}-\frac{\dot{x}}{x}
	=\frac{(\tilde{N}_--2) \dot{x} }{x(x+2-\tilde{N}_-)}\ge \tilde{N}_--2 \;\; \textrm{ for } t\le t_0.}
\end{align*}
By integrating in the interval $[t,t_0]$ we get 
\begin{align*}
\textstyle{ c_0\, \frac{x(t)}{x(t)-(\tilde{N}_--2)} \ge e^{(\tilde N_--2)(t_0-t)}
}\; \Rightarrow\;\;
\textstyle{x(t)\leq -\frac{\tilde{N}_--2 }{ c_0e^{(\tilde{N}_--2)(t-t_0)}-1 }}\,,\; \textrm{ where }
\textstyle{c_0=1-\frac{\tilde{N}_--2}{x(t_0)}>1},
\end{align*}
and in particular $x$ blows up at the finite time $t_1=t_0+\frac{\mathrm{ln}(1/c_0)}{\tilde{N}_--2}<t_0$.
\end{proof}

Regular or singular positive solutions of \eqref{P radial m} and \eqref{P radial m-} enjoy the monotonicity $u^{\prime}<0$ since they belong to  $\overline{1Q}.$ 
Now we obtain a priori bounds for trajectories of \eqref{DS+} or \eqref{DS-} defined for all $t$ in intervals of type $(\hat{t},+\infty)$ or $(-\infty,\hat{t})$.

\begin{prop}\label{AP bounds}
	Let $\tau$ be a trajectory of \eqref{DS+} or \eqref{DS-} in $1Q$, with $\tau(t)=(x(t),z(t))$ defined for all $t\in(\hat{t},+\infty)$, for some $\hat{t}\in \real$. Then $x(t)\in (0, \tilde{N}_\pm-2)$   for all $t\geq \hat{t}$.
	If instead, $\tau$ is defined for all $t\in(-\infty,\hat{t})$, for some $\hat{t}\in \real$, then
	\begin{align}\label{ap z}
	\textrm{$z(t)<\lambda\alpha (N+a) $ in the case of $\M^+$,\quad $z(t)<\Lambda\alpha  (N+a)$ for $\M^-$,\quad for all\; $t\leq \hat{t}$.}
	\end{align}
	In particular, if a global trajectory is defined for all $t\in \real$ in $1Q$ then it remains inside the box $(0,\tilde{N}_+-2)\times (0,\lambda\alpha (N+ a))$ in the case of $\M^+$; it stays in $(0,\tilde{N}_--2)\times (0,\Lambda\alpha (N+ a))$ for $\M^-$.
\end{prop}

\begin{proof}
Since $z>0$, 
$\dot x \ge 
x(x+2-N) \textrm{ in } R_\lambda^+$ and $\dot x \ge 
x(x+2-\tilde{N}_+) \textrm{ in } R_\lambda^-$,
the bound for $x$ for a trajectory defined for all forward time is accomplished as in the proof of \cite[Proposition 2.11]{MNPscalar}. 

Meanwhile, with respect to the bound for $z$, we first claim that if a trajectory $\tau$ intersects the line $z=\lambda (N+a)x$ then the trajectory $\tau$ must cross the $z$ axis. Indeed, in this case $Z$ would attain the value $Z=\lambda (N+a)$, and so a blow up at a backward time $t_0\in \real$ in $Z$ would occur by the proof of (2.26) in \cite[Proposition 2.11]{MNPscalar}. Thus, $X (t)\to 0$ as $t\to t_0^+$, from which $u^\prime (t_0)=0$ with $u(t_0)>0$. Thus, $x(t_0)=0$. 
So the claim is true.
Next, we observe that a trajectory defined for all forward time attains a maximum value for $z$ at the line $\pi_2$, therefore the a priori bound \eqref{ap z} for $z$ is verified.
\end{proof}

\begin{figure}[!htb]\centering	\includegraphics[scale=0.72]{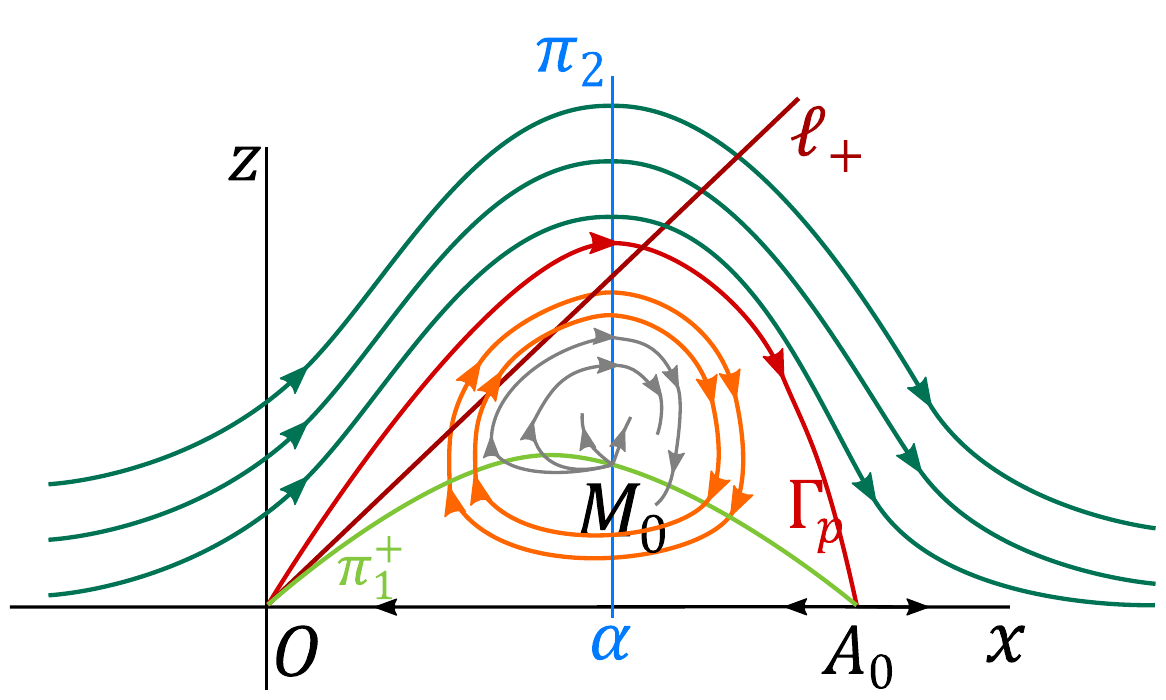}
	\caption{The flow at $p=p^{*}_{a+}$.}
	\label{fig F}
\end{figure}

\begin{figure}[!htb]\centering	\includegraphics[scale=0.54]{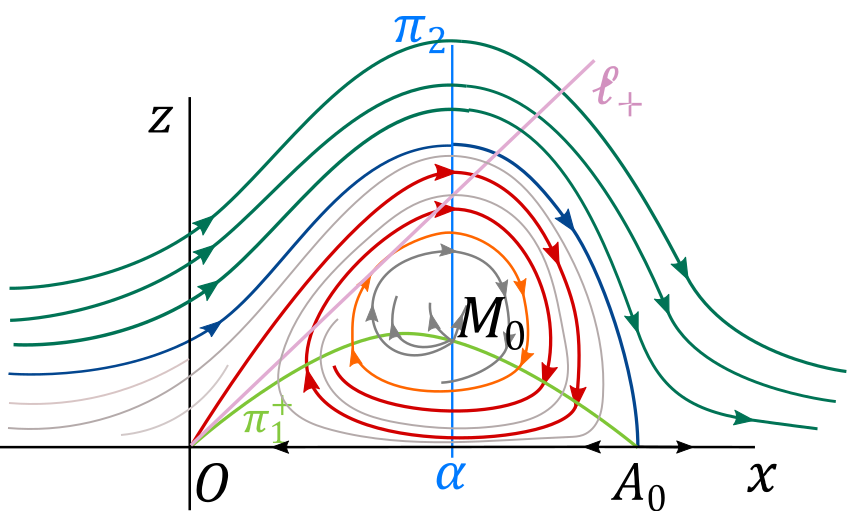}
	\caption{The flow for $p\in (p^{*}_{a+}, p^{p,a}_+)$.}
	\label{fig P}
\end{figure}

\begin{figure}[!htb]\centering	\includegraphics[scale=0.54]{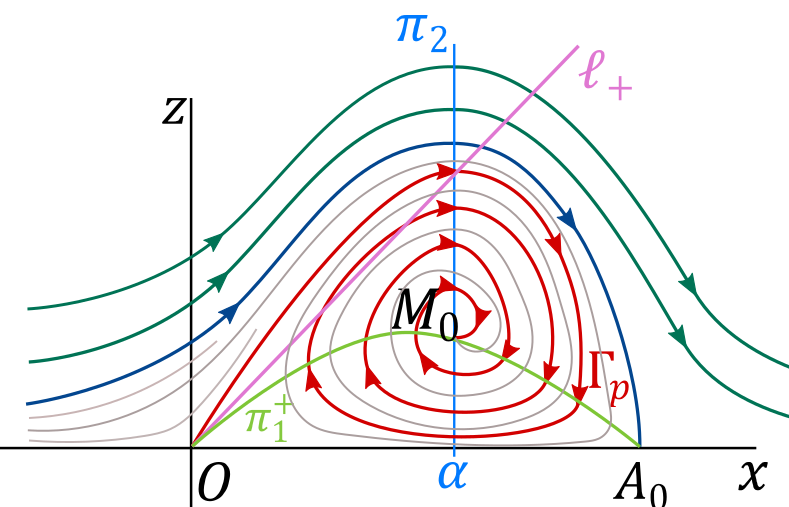}
	\caption{The flow for $p>p^{p,a}_+$.}
	\label{fig S}
\end{figure}

\begin{proof}[Proof of Theorem \ref{teo exterior}] The existence follows by the existence of trajectories produced by the dynamical system in $(x,z)$, which in turn comes from the dynamical system analysis for $(X,Z)$ in \cite{MNPscalar} properly glued via the flow in $2Q$ originated by $x,z$. More precisely, by \cite[Lemma 4.11]{MNPscalar}, for any $p>p^*_{a\pm}$, the orbit $\Upsilon_p$ defined in Proposition \ref{local study stationary points}(2) has a blow-up in $Z$ backwards at finite time $T$, which corresponds to a trajectory in $(x,z)$ which crosses the vertical $z$ axis at $T$. Then, by Lemma \ref{lemma 2Q}, $\Upsilon_p$ has a blow-up in $x$ at $t_0<T$ such that $x(t)\to -\infty$ as $t\to t_0^+$. The trajectory $\Upsilon_p$ corresponds to a fast decaying exterior domain solution of \eqref{P} in $\rN\setminus B_{r_0}$, with $r_0=e^{t_0}$.

The remaining slow decaying and pseudo-slow decaying exterior domain solutions come from the trajectories displayed in \cite[Figures 5(b) and 6]{MNPscalar}, which are again glued through the $z$ axis by our dynamical system $(x,z)$, analogously. 
\end{proof}

Even though solutions in annuli can be identified by the trajectories blowing up both in backward ($x\to -\infty$) and forward ($x\to +\infty$) times in Figures \ref{fig F}--\ref{fig S}, the existence of an annular solution for an arbitrary annulus $(\frak a, \frak b)$ is not ensured. Recall that the scaling in Remark \ref{rescaling} does not work in this case since it changes both extrema of the annulus. This will be accomplished in the next section, by using the shooting method and energy functions.

\section{Energy analysis and solutions in annuli}\label{section annuli}

In this section we introduce some energy functions and use them to establish existence of solutions in the annulus.

\begin{prop}\label{prop E2}
	For each $\delta>0$, and $u_\delta$ solution of \eqref{shooting der}, we set 
	\begin{align}\label{def E2}
	\textstyle  \mathcal{E}_{ \sigma}( r) = \frac{1}{2r^a}(u^\prime )^2 +\frac{1}{\sigma (p+1)}|u|^{p+1} \ \;\;\;\textrm{ for\, } \sigma>0.
	\end{align}
	Then the energy function \vspace{-0.5cm}
	\begin{align*}
	\mathcal{E}( r) =
	\begin{cases}
	\;\mathcal{E}_\Lambda (r)
	\;\;\textrm{ if  \;} uu^\prime>0 \smallskip
	\\
	\;\mathcal{E}_\lambda (r)
	\;\;\textrm{ if  \;} uu^\prime <0 
	\end{cases}
	\end{align*}
	is piecewisely monotone decreasing in $\{u^\prime\neq 0\}$ whenever $\tilde{N}_+\ge 3/2$.
\end{prop}

\begin{proof}
	To fix the ideas let $\delta>0$ and the operator $\M^+$. For simplicity, we write 
	\begin{center}
		$ u^{\prime\prime} +\frac{r^a}{m_{u^{\prime\prime}}}|u|^{p-1}u=-\frac{m_{u^\prime}}{m_{u^{\prime\prime}}}\frac{N-1}{r} \,u^\prime$
	\end{center}
	where $m_{s}$ is the step function defined through $m_{s} s= m_+(s)$, for $s=u^\prime (r)$ or $s=u^{\prime\prime} (r)$, whenever $u^{\prime\prime}\ne 0$. Here $m_+(s)$ is the Lipschitz function given in \eqref{m,M+}. 
	
	Set $\hat{N}-1:=\frac{m_{u^\prime}}{m_{u^{\prime\prime}}}(N-1)$ which is either $N-1$, $\tilde{N}_+-1$ or $\tilde{N}_--1$, whenever $u^{\prime\prime}\ne 0$. We have $\sigma=\Lambda\ge m_{u^{\prime\prime}}$ when $uu^\prime>0$; while $\sigma=\lambda\le m_{u^{\prime\prime}}$ when $uu^\prime <0$. Anyways it yields $\frac{uu^\prime}{\sigma}\le  \frac{uu^\prime}{m_{u^{\prime\prime}}}$, then
	\begin{align*}
	\textstyle \mathcal{E}_\sigma^\prime (r) &= - \textstyle \frac{a}{2} r^{-a-1} (u^\prime)^2 + r^{-a}u^\prime u^{\prime\prime} + \frac{1}{\sigma}\,|u|^{p-1}u u^\prime 
	\\ 
	&\le \textstyle -\frac{a}{2}r^{-a-1} (u^\prime)^2 + r^{-a} u^\prime \,\{  u^{\prime\prime} + \frac{r^a}{m_{u^{\prime\prime}}}\,|u|^{p-1}u  \} 
	\\ 
	&= \textstyle
	-r^{-a-1}  (u^\prime)^2 (\frac{a}{2} +\hat{N}-1 ) < 0 
	\end{align*} 
	whenever $u^{\prime\prime}\neq 0$ and $u^\prime \neq 0$ and $2(\hat{N} -1)+a> 0$. The latter is ensured for instance when $a>-1$ and $\tilde{N}_+\ge 3/2$. 
	Note that at a point $r_0$ where
	$u^{\prime\prime}(r_0)=0$ we have $u^\prime (r_0)$ and $u(r_0)$ with opposite signs since $u^\prime$ and $u$ cannot be both equal to zero by ODE existence and uniqueness of the initial value problem for Lipschitz nonlinearities.
\end{proof}

From the dynamical system we obtain a complete characterization of monotonicity for solutions $u_{\delta}$ of \eqref{shooting der} as follows.
Specially in this section we keep the notation in \cite{GLPradial} for $\tau=\tau_\delta$ as a radius (and not for trajectories as in the rest of the text).
\begin{lem}\label{lema monotonia}
	For any $\delta>0$ such that $u_{\delta}$ is a positive solution of \eqref{shooting der} in $[\frak a, \rho]$, with $\rho=\rho_\delta\le +\infty$, there exists a unique number $\tau=\tau (\delta)$ with $\tau\in (\frak a,\rho)$, such that 
	\begin{center}
		$u^\prime(r)>0 $ \; for $ r\in [\frak a, \tau)$\,, \quad $ u^\prime (\tau)=0$\,, \quad $   u^\prime(r)<0 $ \; for $ r\in (\tau, \rho ]\, .$
	\end{center}
\end{lem}

\begin{proof}
	Let us observe that a first critical point exists for $u$. 
	To see this we look at the dynamical system driven by $X,Z$. In this case, the behavior at the third quadrant $X,Z<0$ is given by $\dot X >0$ and $\dot Z <0$, with a blow up at finite time $T$ such that $u^{\prime }(e^T)=0$ by \cite[Remark 3.3]{MNPscalar}, and so $\tau=e^T$.
	At such time we have $u(\tau)>0$, and so $X(T)=x(T)=0$ and $z(T)\in (0,+\infty)$.
	The uniqueness of $\tau$ follows by Remark \ref{rmk maximum point}.
\end{proof}  

If $\rho_\delta=+\infty$ then $\lim_{r\to \infty} u_\delta(r)=0$.
This comes from the a priori bounds in Proposition \ref{AP bounds}.
Thus, for any $\delta>0$ either  $\rho_{\delta}=+\infty$ and $\lim_{r\to \infty} u_{\delta}(r)=0$, or there exists some $\rho_{\delta}<+\infty$ such that $u(\rho_{\delta})=0$. Moreover,  by continuous dependence on the initial data, the function $\delta \mapsto \rho_\delta$ is continuous in a neighborhood  of any $\delta>0$ where $\rho_\delta<+\infty$ whenever $p> 1$.

We shall omit the dependence on the parameter $\delta>0$ whenever it is clear from the context.

\begin{prop}\label{prop energies}
	For any pair $\delta>0$, and $u$ of \eqref{shooting der}, the energy functions
	\begin{align*}
	\textstyle E_\lambda ( r)  &=
	\,  r^{2(\tilde N_{-}-1)+a}\; \mathcal{E} _\lambda (r)\;\;\textrm{ in }\; [\frak a, \tau ]
	\medskip
	\\
	\textstyle E_\Lambda ( r)  &=	\, r^{2(\tilde N_{-}-1)+a}\; \mathcal{E}_\Lambda (r)
	\;\;\textrm{ in }\; [\tau, \rho]
	\end{align*}
	are monotone increasing, where $\mathcal{E}_\sigma$ is given in \eqref{def E2} for $\sigma\in \{\lambda, \Lambda\}$. 
\end{prop}

\begin{proof} Let us consider the operator $\M^+$.
	We recall that in the interval $[\frak a, \tau ]$ we have $u^\prime \ge 0$, $u^{\prime\prime}\le 0$, and so
	\begin{center}
		$	u^{\prime\prime}u^\prime +\frac{r^a}{\lambda} u^p u^\prime = - \frac{(\tilde{N}_--1) }{ r} (u^\prime )^2 .	$
	\end{center}
	On the other hand, in $[\tau, \rho]$ we have $u^\prime \le 0$, and 
	\begin{center}
		$u^{\prime\prime}u^\prime +\frac{r^a}{\Lambda} u^p u^\prime \ge u^{\prime\prime}u^\prime  +\frac{r^a}{m_{u^{\prime\prime}}}u^p u^\prime= - \frac{(\hat{N} -1) }{r} (u^\prime )^2 \ge - \frac{(\tilde{N}_--1)}{ r} (u^\prime)^2$,
	\end{center}
	where $(m_{u^{\prime\prime}}, \hat{N})$ is either $(\lambda, N)$ or $(\Lambda, \tilde{N}_+)$.\smallskip
	
	Set $\sigma=\lambda$ if $r\in [\frak a , \tau]$ and  $\sigma=\Lambda$ if $r\in [ \tau, \rho]$. In any case, for $A=2(\tilde{N}_- -1)$ we obtain
	\begin{align*}
	\textstyle E_\sigma^\prime (r) 
	&= \textstyle  Ar^{A-1} 
	\left\{ \frac{1}{2} (u^\prime)^2  +\frac{r^a}{\sigma (p+1)} u^{p+1}  
	\right\}
	+r^A
	\left\{ u^{\prime\prime} u^\prime  +\frac{r^a }{\sigma} u^{p}u^\prime +\frac{ar^{a-1}}{\sigma (p+1)} u^{p+1}
	\right\} \smallskip \\
	&\textstyle \ge r^{A-1} (u^\prime)^2\, \{ \frac{A}{2}- (\tilde{N}_- -1) \}
	= 0
	\end{align*}
where the inequality comes from $2(\tilde{N}_- -1)+a\ge 0$, which holds for $a>-1$ and $N\ge 2$.
\end{proof}

In the remaining of the section we prove Theorem \ref{teo anel}. 
This is  reduced to show that, for any given $+\infty>\frak b> \frak a>0$,  there exists a parameter $\delta>0$ such that $\rho_{\delta}=\frak b$ in addition to $u(\frak b)=0$.

We start analyzing the behavior of the solutions $u_\delta$ as $\delta$ approaches to $0$ and $+\infty$.

\begin{lem}\label{lemma delta 0}
If $\delta \to 0$ then we have
	$u (\tau_\delta)\to 0$ and 
	$\rho_{\delta}\to +\infty \, .$
\end{lem}

\begin{proof}
	By Proposition \ref{prop E2} we have $\mathcal{E}_\Lambda(r) \leq \mathcal{E}_\Lambda (\frak a)$ for all $r\le \tau$, that is,
	\begin{equation}\label{eq_bound above}
	\textstyle \frac{{\frak a}^a}{p+1}u^{p+1}(r) \le \frac{\Lambda}{2} \,\delta^2 \quad \textrm{ for all } \; r\in [\frak a, \tau],
	\end{equation}
	since $uu^\prime  \ge 0$ in $[\frak a, \tau]$.
	In particular, at $r=\tau=\tau_\delta$\,,
	\begin{equation*}
	\textstyle u^{p+1}(\tau_\delta)\le \frac{\Lambda (p+1)}{2 {\frak a}^a} \,\delta^2 \to 0 \;\;\textrm{ when  } \delta\to 0.
	\end{equation*}
	
	Next we write the equation for $u$ in $[\frak a, \tau]$ as
	$(u^\prime r^{\tilde{N}_- -1})^\prime = - \frac{r^a}{\lambda} u^p \,r^{\tilde N_- -1}$, and so integrating from $\frak a$ to $\tau$ produces
	\begin{align}\label{cara mu}  
	0=u^\prime (\tau)\, \tau^{\tilde N_- -1}=\delta \, {\frak a}^{\tilde N_- -1}- \frac{1}{ \lambda  }\int_{\frak a}^{\tau} s^{\tilde N_- -1+a}\, {u^p} \, .
	\end{align}

	By combining the estimate for $u$ in \eqref{eq_bound above} and equality \eqref{cara mu} we obtain
	\begin{align*}
	\delta =\frac{1}{\lambda {\frak a}^{\tilde N_- -1}} \int_{\frak a}^{\tau} s^{\tilde N_- -1+a}\, u^p\,
	\le\,
	{C_0} \,\delta^{\frac{2p}{p+1}} \; \tau^{\tilde N_-+a}
	\end{align*}
where $C_0 $ depends only on $\frak a,p, a, N,\lambda, \Lambda$,	and so
	\begin{align*}
	\tau_\delta^{\tilde N_-+a} \,\ge \,\frac{1 }{C_0\,\delta^{\frac{p-1}{p+1}}} \to +\infty \;\;\textrm{ as } \delta \to 0 \, .
	\end{align*}
	In particular, $\rho_{\delta}\ge \tau_\delta \to +\infty$ as $\delta\to 0$.
\end{proof}

\begin{lem}\label{lema delta infty}
If $\delta \to +\infty$ then  $\rho_\delta\to \frak a  $ and 
	$u (\tau_\delta)\to +\infty$.
	Moreover, for every $C_0>0$  there exists a positive constant $c_0$ depending only on $C_0, \frak a, N,p,\lambda,\Lambda$ such that
	\begin{center}
		$\delta \le C_0$ \quad implies \quad  $\rho_\delta \ge \frak a +c_0$.
	\end{center}
\end{lem}

\begin{proof} We denote $A_{\frak r, \frak s}=B_{\frak r}\setminus \overline{B}_{\frak s}$ for any $\frak r >\frak s$ and fix the operator $\M^+$, for $\M^-$ it will be analogous.
	
	Step 1) $u(\tau_\delta)\to +\infty$ when $\delta \to \infty$.
	
	Assume by contradiction that there exists a sequence $\delta_k\to \infty$ with respective solutions $u_k=u_{\delta_k}$ of \eqref{shooting der}, with $\tau_k=\tau_{\delta_k}$\,, $\rho_k=\rho_{\delta_k}$\,, and  $u_k\le M$ for all $k$.
	
	Since $E_\lambda(r)\ge E_\lambda(\frak a)$ for all $r\in [\frak a , \tau_k]$ by Proposition \ref{prop energies}, then
	\begin{align}\label{bound below}
	\textstyle \tau_k^{2(\tilde{N}_- -1)+a}\; u_k^{p+1}(\tau_k)\ge \frac{\lambda(p+1)}{2} {\frak a}^{2(\tilde{N}_- -1)}\; \delta_k^2 \,\to +\infty.
	\end{align}
	Since $u_k\le M$, then $\tau_k \to \infty$ as $k\to \infty$. In particular, $\tau_k\ge \frak a +1$ for large $k$. Take $\varepsilon\in (0,1)$ with  $\varepsilon \le \frac{\frak a}{\tilde{N}_- -1}$ and $r\in [\frak a, \frak a+\varepsilon]\subset [\frak a , \tau_k]$. Then we use Taylor expansion of $u_k$ at the point $\frak a$ to write
	\begin{align}\label{taylor}
	\textstyle u_k(r) = u_k(\frak a) + u^\prime_k(\frak a)(r- \frak a) + \frac{1}{2}\, u^{\prime\prime}_k(c_k)(r- \frak a)^2 , \;\;\textrm{ for some $ c_k \in (\frak a, r)$.}
	\end{align}
	
	Now we notice that
	\begin{center}
		$\frac{\tilde{N}_- -1}{c_k} u^\prime_k(c_k)  \leq  \frac{\tilde{N}_- -1}{\frak a} \delta_k$
		\quad since\; $ u^\prime_k(c_k) \in (0,\delta_k)$.
	\end{center}
	Moreover, since $u_k^\prime$ is decreasing in $(\frak a, r)$, we have 
	$u^\prime_k(c_k) \leq \delta_k $ and so, by the second order PDE in \eqref{P radial m} and the fact that $u_k$ is increasing in $(\frak a, r)$, we deduce
	\begin{align*}
	\textstyle u^{\prime\prime}_k(c_k)=- \frac{\tilde{N}_- -1}{c_k} u^\prime_k(c_k) - \frac{c_k^a}{\lambda} u_k^p(c_k) \geq - \frac{\tilde{N}_- -1}{\frak a} \delta_k - \frac{(\frak a +1)^a}{\lambda} u_k^p(r).
	\end{align*}
	Putting this estimate into \eqref{taylor} one finds
	\begin{align*}
	\textstyle u_k(r) \geq \delta_k(r-\frak a) - \frac{\tilde{N}_- -1}{2 \frak a} \delta_k(r-\frak a)^2 - \frac{(\frak a +1)^a}{2 \lambda}u_k^p(r) \,(r-\frak a)^2.
	\end{align*}
	Finally,  by evaluating it at $r=\frak  a+\varepsilon$ it yields 
	\begin{align*}
	\textstyle 
	u_k(\frak a+\varepsilon)+  \frac{ (\frak a +1)^a\varepsilon^{2} }{2\lambda} \, u_k^p(\frak a+\varepsilon)
	\ge \, \delta_k \varepsilon \left\{ 1 - \frac{\tilde{N}_- -1}{2 \frak a} \varepsilon  \right\} \ge \frac{1}{2}\delta_k \varepsilon \;\;
	\textrm{ for sufficiently large } k,
	\end{align*}
	since $ \frac{\tilde{N}_- -1}{2 \frak a} \varepsilon \le \frac{1}{2}$. But this is impossible since $\delta_k\to +\infty$ and $u_k$ is bounded. This shows Step 1.

	\medskip
	
	Step 2) $\rho_\delta \to \frak a$ as $\delta \to +\infty$.
	
	We first show that $\tau_\delta \to \frak a $  as $\delta \to +\infty$. This will be a consequence of Step 1 and the estimate
	\begin{align}\label{conta feia}
	\textstyle u_\delta^{\frac{p-1}{2}} (\tau_\delta) \le {C_0}\,(\tau_\delta -\frak a)^{-1} ,
	\end{align}
where $C_0$ depends on $\frak a, a, p, \lambda$.
	In order to prove \eqref{conta feia}, we write for $r\in [\frak a, \tau]$ where $\tau=\tau_\delta$,
	\begin{center}
		$- \frac{1}{2}( \,u^\prime(r)^2 )^\prime\, \ge -u^{\prime\prime} u^\prime - \frac{\tilde{N}_- -1}{r} (u^\prime)^2 \ge \frac{r^a}{\lambda} u^p u^\prime \ge \frac{{\frak a}^a}{\lambda} u^p u^\prime$,
	\end{center}
	and by integrating it in $[r, \tau]$, for $r\in [\frak a, \tau)$, one gets
	\begin{center}
		$u^\prime (r)\, \ge\, C \, \{ u^{p+1}(\tau)-u^{p+1}(r) \}^{1/2}$,
	\end{center}
where $C$ depends on $\frak a, a, p, \lambda$.
Another integration in $[\frak a, \tau]$ yields
	\begin{align*}
	\textstyle	\int_{\frak a}^\tau \frac{u^\prime\, \rmd r }{\sqrt{u^{p+1}(\tau)-u^{p+1}(r)}} \,\geq\, C \int_{\frak a}^\tau \rmd r\, =\, C (\tau - \frak a).
	\end{align*}
	By using $s = u(r)$ and $u^\prime \rmd r = \rmd s$ we get
	\begin{align*}
	\textstyle	C (\tau - \frak  a)\le \int_0^{u(\tau)} \frac{\rmd s }{\sqrt{u^{p+1}(\tau)-s^{p+1}}} =   \frac{1}{u^{\frac{p+1}{2}}(\tau)} \;\int_0^1 \frac{u(\tau)\, \rmd \sigma}{\sqrt{1-\sigma^{p+1}}} = \frac{1}{u^{\frac{p-1}{2}}(\tau)} \int_0^1 \frac{ \rmd \sigma}{\sqrt{1-\sigma^{p+1}}}
	\end{align*}
	by taking $\sigma = \frac{s}{u(\tau)}$ and $\rmd\sigma = \frac{\rmd s}{u(\tau)}$, from which we deduce \eqref{conta feia}.
	
	\smallskip
	
	Now it is enough to prove that \begin{center}
		$\displaystyle\lim_{\delta \to \infty} \frac{\rho_{\delta}}{\tau_\delta}=1$.
	\end{center}
	If not, then there exists $\epsilon>0$ and a sequence $\delta_k\to \infty$, with $\rho_k=\rho_{\delta_k} \le +\infty$ and $\tau_k=\tau_{\delta_k}$ such that 
	$\rho_k>(1+\epsilon)\tau_k$ for the solutions $u_k=u_{\delta_k}$  of \eqref{shooting der}.
	In particular, $u_k$ is positive and decreasing in the interval $[\tau_k,(1+\epsilon)\tau_k]$.
	
	For $r\in (\tau_k,(1+\epsilon)\tau_k]$ we consider the annulus $A_k=A_{\tau_k, r}$ where $u_k$ solves
	\begin{center}
		$-\mathcal{M}^\pm (D^2 u_k) \ge  \, t_k |x|^a u_k $ \; in $A_k$ , \quad $u_k>0$ in $ A_k$,
	\end{center}
	where
	\begin{center}
		$t_k=\min_{A_k} u_k ^{p-1}=u_k^{p-1}(r)$.
	\end{center}
	
	Now, by the definition of first eigenvalue $\lambda_1^+({D})=\lambda_1^+(\M^+,{D})$  for the fully nonlinear Lane-Emden equation driven by $\M^+$ in the domain ${D}$ with respect to the weight $|x|^a$ (see \cite{BEQPucciradial, pq1, BQeq}), we have
	\begin{align}\label{bound eigenvalue def}
	\textstyle u_k^{p-1}(r)\le \lambda_1^+( A_k)
	,  \;\;\textrm{ for all } r\in (\tau_k \, , (1+\epsilon)\tau_k).
	\end{align}

	Note that the following scaling holds
	\begin{align}\label{scaling autov}
	\textrm{$\lambda_1^+(A_{\frak s ,\, \frak s (1+\epsilon)}\,) = \frac{1}{ {\frak s}^{2+a}} \,\lambda_1^+ (A_{1,1+\epsilon})$, \;\; for all $\frak s >0$.}
	\end{align}
	In fact, if $\lambda_1^+,\phi_1^+$ are a positive eigenvalue and eigenfunction for the operator $\M^+$ with weight $|x|^a$ in $A_{1,1+\epsilon}$ i.e.\
	\begin{center}
		$\M^+(D^2 \phi_1^+)+\lambda_1^+ |x|^a \phi_1^+=0$, \;\;  $\phi_1^+>0$ \;\; in $A_{1,1+\epsilon}$, \quad $\phi_1^+=0$ \; on\, $\partial A_{1,1+\epsilon} $
	\end{center}
	then $\mu_1^+, \psi_1^+$, where $\mu_1^+={\lambda_1^+}{\frak{s}^{-2-\frak a}}$\, and $\psi_1^+(x)=\phi_1^+(\frac{x}{\frak s})$ are a positive eigenvalue and eigenfunction in $A_{\frak s,\frak s (1+\epsilon)}$ for $\M^+$ with weight $|x|^a$.
	
	\smallskip
	
	Then, by combining \eqref{bound eigenvalue def} and \eqref{scaling autov} one finds \begin{align}\label{bound u}
	\textstyle u_k^{p-1}(r)\le\, 
	\frac{1}{{\frak a }^{2+a}}\,
	\lambda_1^+( A_{1,1+\frac{\epsilon}{2}})
	,  \;\;\textrm{ for all } r\in [(1+\frac{\epsilon}{2})\tau_k, (1+\epsilon)\tau_k].
	\end{align}

	\smallskip
	
	Using $E_\Lambda(\tau_k)\le E_\Lambda (r)$ for $r\in [\tau_k, \rho_k)$, it comes
	\begin{align}\label{est der}
	\textstyle r^{2(\tilde{N}_--1)}
	\left\{\frac{r^a}{\Lambda (p+1)} u_k^{p+1}(r)+\frac{1}{2}(u_k^\prime )^2 (r) \right\}
	&\ge\, \textstyle
	\frac{\tau_k^{{2(\tilde{N}_- -1)+a}}}{\Lambda (p+1)} \,u_k^{p+1}(\tau_k)\nonumber \\
	&\ge\, \textstyle 
	\frac{{\frak {a}}^{2(\tilde{N}_- -1)+a}}{\Lambda (p+1)} \,u_k^{p+1}(\tau_k).
	\end{align}
	
	Since $\tau_k\to \frak a$ as $k\to +\infty$ then \begin{center}
		$r\le (1+\epsilon)\tau_k\le  (1+\epsilon)(\frak a +1)$ \; for large $k$.
	\end{center} Now, by putting the latter and \eqref{bound u} into \eqref{est der} we derive 
	\begin{center}
		$(u_k^\prime )^2 (r) \ge J_k $\,,
	\end{center}
	where
	\begin{center}
		$J_k:=C_{\epsilon, p,N,\Lambda}  \{ 
		(\frak a +1 )^{-2(\tilde{N}_- -1)}{\frak a}^{2(\tilde{N}_- -1)+a}\, u_k^{p+1}(\tau_k) -(\frak a +1)^{a}
		{\frak{a}}
		^{\frac{-(2+a)(p+1)}{p-1}}
		\}$
	\end{center} 
	and $J_k\to + \infty $ as $k\to \infty$ by Step 1. Hence 
	\begin{center}
		$-u_k^\prime (r)\ge J_k^{1/2}\to +\infty $ as $k\to \infty$, \; for all\, $r\in [(1+\frac{\epsilon}{2})\tau_k \, , (1+\epsilon)\tau_k]$.
	\end{center} 
	Via integration we get
	\begin{align*}
	u_k((1+{\textstyle{\frac{\epsilon}{2}}})\tau_k)&\ge\, u_k((1+{\textstyle{\frac{\epsilon}{2}}})\tau_k)-u_k((1+{\epsilon})\tau_k)\\
	&=\,
	-\int_{(1+\epsilon/2)\tau_k}^{(1+\epsilon)\tau_k} u_k^\prime (r) \rmd r \ge \textstyle{\frac{\epsilon \tau_k}{2}}
	J_k^{1/2} \to +\infty
	\end{align*}
	which contradicts \eqref{bound u}.
	
	\smallskip
	
	Step 3)  $\delta \le C_0$ implies $\rho_\delta \ge \frak a +c_0$.
	
	Let us prove the contrapositive, that is, if $\rho_\delta \to \frak a$ then $\delta \to +\infty$.
	
	As in Step 2, if $s=\max_{A_{\frak a , \rho}} u^{p-1} =u^{p-1}(\tau)$ then $u$ solves
	\begin{center}
		$-\mathcal{M}^\pm (D^2 u) \le   |x|^a u^p\le s\,|x|^a u$  \; \; in\, $A_{\frak a , \rho}$\,, \; $u= 0$ \; on\, $\partial A_{\frak a , \rho}$\,.
	\end{center}
	
	Now, by the maximum principle for the fully nonlinear equation through the characterization of the first eigenvalue in \cite{BEQPucciradial, BQeq} (see also \cite{pq1} for the weighted version) it follows
	\begin{align}\label{cota autovalor}
	u^{p-1}(\tau) \ge \lambda_1^+ (\M^+, A_{\frak a , \rho}\,).
	\end{align}
	In fact, if we had $s<\lambda_1^+ (\M^+, A_{\frak a , \rho}\,)$ then by the mentioned maximum principle we would obtain $u\le 0$ in $A_{\frak a , \rho}$ which is impossible.
	
	Using the scaling for the eigenvalue in \eqref{scaling autov}, \eqref{eq_bound above}, and \eqref{cota autovalor}, we derive
	\begin{center}
		$\lambda_1^+(\M^+, A_{1,\, { \rho}/{\frak a} }\, ) = {\frak a}^{2+a} \, \lambda_1^+ (\M^+, A_{\frak a , \rho}\,) \le  
		{\frak a}^{2+a} \, (\,\frac{\Lambda (p+1)}{2\frak{a}^a}\,\delta^2\,)^{\frac{p-1}{p+1}}.
		$
	\end{center}
	Again by the scaling as in Step 2, $\lambda_1^+ (\M^+,D) \to +\infty$ as $|D|\to 0$. Then $\rho_\delta \to \frak a$ implies $\delta \to +\infty$.
	As a consequence, the ratio ${\rho_\delta}/{\frak a}$ remains bounded away from 1 whenever $\delta$ is bounded from above.
\end{proof}

\begin{proof}[Proof of Theorem \ref{teo anel}.]
	We fix the annulus $A_{\frak a , \frak b}$ for some $0<\frak a<\frak b$. 
	For every $\delta>0$, recall that $u_\delta $ is the unique radial solution of the initial value problem \eqref{shooting der}, with a maximal radius of positivity given by $\rho_\delta \in (\frak a, +\infty]$.
	Here, $u(\rho_\delta)=0$ if $\rho_\delta <+\infty$, while $u(r)\to 0$ as $r\to +\infty$ is $\rho_\delta=+\infty$.
	
	The mapping $\delta \to \rho _\delta$ is continuous by ODE continuous dependence on initial data. In particular, the set 
	\begin{align}\label{def conjunto D}
	\mathcal{D}=\mathcal{D}\,(p):=\{ \,\delta\in (0,+\infty) : \; \rho_\delta <+\infty \, \}
	\end{align}
	is open. By Lemma \ref{lema delta infty}, $\mathcal{D}$ is nonempty and contains an open neighborhood of $+\infty$.
	
	Let $\delta^*=\delta^*(p)$ be the infimum of the unbounded connected component of $\mathcal{D}$. 
	Since $\mathcal{D}$ is open, if $\delta^*>0$ then $\rho_{\delta^*}=+\infty$. If $\delta^*=0$ then $\lim_{\delta \to 0} \rho_\delta \ge \lim_{\delta \to 0} \tau_\delta=+\infty$ by Lemma \ref{lemma delta 0}.
	
	The function $\delta \mapsto \rho_\delta$ is well defined and leads the interval $(\delta^*, +\infty)$ onto $(\frak a,  +\infty)$ by the second part of Lemma \ref{lema delta infty}. Then there exists $\delta>0$ such that $\rho_\delta=\frak b$.
	The existence of negative solutions follows by Remark \ref{rmk G}.
\end{proof}

\begin{rmk}
	For the non weighted case $a=0$, in \cite{GILexterior2019} it was shown that $\delta^*=\inf \mathcal{D}$ for all $p$, that is $\mathcal{D}=(\delta^*, +\infty)$ is an open interval. Moreover, they prove there that at $\delta^*$ only a fast decaying solution is admissible.
\end{rmk}

\textbf{{Acknowledgments.}}
The authors would like to thank Carmem Maia Gilardoni for kindly making the figures, and for the precise comments of the anonymous referee.

L.\ Maia was supported by FAPDF, CAPES, and CNPq grant 309866/2020-0. G.\ Nornberg was supported by FAPESP grant 2018/04000-9, São Paulo Research Foundation.


\end{document}